\documentclass[11pt]{amsart}

\usepackage{amsmath, pb-diagram}
\usepackage{amssymb}
\usepackage{amscd}
\usepackage[all]{xy}
\begin{document}

\title{Dual $\pi$-Rickart Modules}
\author{B. Ungor}
\address{Burcu Ungor, Department of Mathematics, Ankara University, Turkey}
\email{bungor@science.ankara.edu.tr}

\author{Y. Kurtulmaz}
\address{Yosum Kurtulmaz, Department of Mathematics, Bilkent University,  Turkey}
\email{yosum@fen.bilkent.edu.tr}

\author{S. Hal\i c\i oglu}
\address{Sait Hal\i c\i oglu,  Department of Mathematics, Ankara University, Turkey}
\email{halici@ankara.edu.tr}

\author{A. Harmanci}
\address{Abdullah Harmanci, Department of Maths, Hacettepe University, Turkey}
\email{harmanci@hacettepe.edu.tr}

\date{}
\newtheorem{thm}{Theorem}[section]
\newtheorem{lem}[thm]{Lemma}
\newtheorem{prop}[thm]{Proposition}
\newtheorem{cor}[thm]{Corollary}
\newtheorem{exs}[thm]{Examples}
\newtheorem{defn}[thm]{Definition}
\newtheorem{nota}{Notation}
\newtheorem{rem}[thm]{Remark}
\newtheorem{ex}[thm]{Example}

\begin{abstract}

Let $R$ be an arbitrary ring with identity and $M$ a right
$R$-module with $S =$ End$_R(M)$. In this paper  we  introduce
dual $\pi$-Rickart modules as a generalization of $\pi$-regular
rings  as well as that of dual Rickart modules. The module $M$ is
said to be {\it dual $\pi$-Rickart} if for any $f\in S$, there
exist $e^2=e\in S$ and a positive integer $n$ such that
Im$f^n=eM$. We prove that some results of dual Rickart modules can
be extended to dual $\pi$-Rickart modules for this general
settings. We investigate  relations between a dual $\pi$-Rickart
module and its endomorphism ring.

 \vspace{2mm}

 \noindent {\bf Key
words}:  $\pi$-Rickart modules, dual $\pi$-Rickart modules,
Fitting modules, generalized left principally projective rings,
$\pi$-regular rings.
\end{abstract}

\maketitle

\section{Introduction}

\indent\indent Throughout this paper $R$ denotes an associative
ring with identity, and modules are unitary right $R$-modules.
 For a module $M$,  $S=$ End$_R(M)$ is the ring of all right $R$-module endomorphisms of
$M$. In this work, for the $(S, R)$-bimodule $M$,  $l_S(.)$ and
$r_M(.)$ are the left annihilator of a subset of $M$ in $S$ and
the right annihilator of a subset of $S$ in $M$, respectively. A
ring is {\it reduced} if it has no nonzero nilpotent elements.
{\it Baer rings} \cite{Ka} are introduced as rings in which the
right (left) annihilator of every nonempty subset is generated by
an idempotent. Principally projective  rings were introduced by
Hattori \cite{Hat} to study the torsion theory, that is,  a ring
$R$ is called  {\it left (right) principally projective} if every
principal left (right) ideal is projective. The concept of left
(right) principally projective rings (or left (right) Rickart
rings) has been comprehensively studied in the literature.
Regarding a generalization of Baer rings as well as
 principally projective rings, recall that a ring $R$ is called
{\it generalized left (right) principally projective} if for any
$x\in R$, the left (right) annihilator of $x^n$ is generated by an
idempotent for some positive integer $n$. A number of papers have
been written on generalized principally projective rings (see
\cite{Hir} and  \cite{HKL}).  A ring $R$ is  \emph{(von Neumann)
regular} if for any $a \in R$ there exists $b \in R$ with $a =
aba$. The ring $R$ is called {\it $\pi$-regular} if for each $a\in
R$ there exist a positive integer $n$ and an element $x$ in $R$
such that $a^n=a^nxa^n$. Similarly, call a ring $R$ {\it strongly
$\pi$-regular} if for every element $a \in R$ there exist a
positive integer $n$ (depending on $a$) and an element $x\in R$
such that $a^n = a^{n+1} x$, equivalently, there exists $y\in R$
such that $a^n=ya^{n+1}$. Every regular ring is $\pi$-regular and
every strongly $\pi$-regular ring is $\pi$-regular.
 There are regular or $\pi$-regular rings which are not strongly $\pi$-regular.

According to Rizvi and Roman, a module $M$  is said to be {\it
Rickart} \cite{TR3}  if for any $f\in S$, $r_M(f)=eM$ for some
$e^2=e\in S$. The class of Rickart modules is studied extensively
by different authors (see \cite{AHH3} and \cite{LRR}). Recently
the concept of a Rickart module is generalized in \cite{UHH} by
the present authors. The module $M$ is called {\it $\pi$-Rickart}
if for any $f\in S$, there exist $e^2=e\in S$ and a positive
integer $n$ such that $r_M(f^n)=eM$.  Dual Rickart modules are
defined by Lee, Rizvi and Roman in \cite{LRR1}. The module $M$ is
called {\it dual Rickart} if for any $f\in S$, Im$f=eM$ for some
$e^2=e\in S$.

In the second section,  we investigate general properties of dual
$\pi$-Rickart modules and Section 3 contains the results on the
structure of endomorphism ring of a dual $\pi$-Rickart module. In
what follows, we denote by $\Bbb Z$, $\Bbb Q$, $\Bbb R$ and $\Bbb
Z_n$ integers, rational numbers, real numbers and the ring of
integers modulo $n$, respectively, and $J(R)$ denotes the Jacobson
radical of a ring $R$.

\section{dual $\pi$-Rickart Modules} In this section, we introduce
 the concept of a dual $\pi$-Rickart
module that generalizes the notion of a dual Rickart module as
well as that of a $\pi$-regular ring.  We prove that some
properties of dual Rickart modules hold for this general setting.
Although  every direct summand of a dual $\pi$-Rickart module is
dual $\pi$-Rickart, a direct sum  of dual $\pi$-Rickart modules is
not dual $\pi$-Rickart. We give an example to show that a direct
sum of dual $\pi$-Rickart modules may not be dual $\pi$-Rickart.
It is shown that the class of some abelian dual $\pi$-Rickart
modules is closed under direct sums.

We start with our main definition.

\begin{defn} {\rm Let $M$ be an $R$-module with $S=$ End$_R(M)$. The module $M$ is
called {\it dual $\pi$-Rickart} if for any $f\in S$, there exist
$e^2=e\in S$ and a positive integer $n$ such that Im$f^n=eM$. }
\end{defn}
For the sake of brevity, in the sequel, $S$ will stand for the
endomorphism ring of the module $M$ considered. Dual $\pi$-Rickart
modules  are abundant around. Every semisimple module, every
injective module over a right hereditary ring and every module of
finite length are dual $\pi$-Rickart. Also every quasi-projective
strongly co-Hopfian module, every quasi-injective strongly Hopfian
module, every Artinian and Noetherian module is dual $\pi$-Rickart
(see Corollary \ref{fgric}). Every finitely generated module over
a right Artinian ring is a dual $\pi$-Rickart module (see
Proposition \ref{artin}).

\begin{prop}\label{pi} Let $R$ be a ring. Then the right $R$-module
$R$ is a dual $\pi$-Rickart module if and only if $R$
is a $\pi$-regular ring.
\end{prop}
\begin{proof} If the right $R$-module $R$ is a dual $\pi$-Rickart
module and $f \in R$, then there
exist $e^2=e\in R$ and a positive integer $n$ such that
Im$f^n=eR$. There exist $x,y\in R$ such that $e=f^{n}x$ and
$f^{n}=ey$. Multiplying the first equation from the right by
$f^{n}$,  we have  $f^{n}xf^{n}=ey=f^{n}$. Conversely, assume that
$R$ is a $\pi$-regular ring. Let $g \in R$. Then there exist a
positive integer $n$ and   $x\in R$ such that $g^{n}=g^{n}xg^{n}$.
Hence $e=g^{n}x$ is an idempotent of $R$. Since $e \in g^{n}R$ and
$g^{n}=g^{n}xg^{n}=eg^{n}\in eR$, we have Im$g^{n}=eR$. Therefore
the right $R$-module $R$ is  dual $\pi$-Rickart.
\end{proof}

It is clear that every dual Rickart module is dual $\pi$-Rickart.
The following example shows that every dual $\pi$-Rickart module
need not be dual Rickart.

\begin{ex}\rm{ Let $R$ denote the ring $\left (\begin{array}{cc}\Bbb Z_2 & \Bbb{Z}_2\\0 & \Bbb Z_2\end{array}\right)$
and  $M$  the right $R$-module  $\left (\begin{array}{cc}0 &
\Bbb{Z}_2\\\Bbb Z_2 & \Bbb Z_2\end{array}\right)$ with usual
matrix operations. If  $f\in S =$ End$_R(M)$, then there exist
$a$, $b$, $c\in \Bbb Z_2$ such that
\begin{center}$f\left (\begin{array}{cc} 0 & x\\y & z \end{array}\right) = \left (\begin{array}{cc} 0 & ax\\
by & cx+bz \end{array}\right)$ \end{center} \noindent By using
this image of $f$,  we prove that there exists a positive integer
$n$ such that Im$f^n$ is a direct summand of $M$. Consider
the following cases for $a$, $b$, $c\in \Bbb Z_2$.\\
$\texttt{Case 1}$.  If  $a = b = c = 1$, then $f$ is an epimorphism.\\
$\texttt{Case 2}$.  If $a = 0$, $b = 0$, $c = 1$, then  $f^2 = 0$.\\
$\texttt{Case 3}$.  If $a = 0$, $b = 1$, $c = 1$ or $a = 0$, $b =
1$, $c = 0$, then in either case Im$f = \left\{\left
(\begin{array}{cc} 0 & 0\\x & y
\end{array}\right)\mid x, y\in \Bbb Z_2\right\}$ is a direct summand
of $M$.\\
$\texttt{Case 4}$.  If $a = 1$, $b = 0$, $c = 1$, then Im$f =
\left\{\left (\begin{array}{cc} 0 & x\\0 & x
\end{array}\right)\mid x\in \Bbb Z_2\right\}$ is a direct summand of
$M$.\\
$\texttt{Case 5}$. If $a = 1$, $b = 0$, $c = 0$, then Im$f =
\left\{\left (\begin{array}{cc} 0 & x\\0 & 0
\end{array}\right)\mid x\in \Bbb Z_2\right\}$ is a direct summand of
$M$.\\
$\texttt{Case 6}$. If $a = 1$, $b = 1$, $c = 0$, then $f$ is an identity map. \\
$\texttt{Case 7}$. If $a = 0$, $b = 0$, $c = 0$, then $f$ is a
zero map.

\noindent In all cases there exists a positive integer $n$ such
that  Im$f^n$ is a direct summand of $M$ and so $M$ is a dual
$\pi$-Rickart module. The module $M$ is not dual Rickart by the
second case, since  Im$f = \left\{\left (\begin{array}{cc} 0 &
0\\0 & x
\end{array}\right)\mid x\in \Bbb Z_2\right\}$. }\end{ex}

Our next aim is to find conditions under which a dual
$\pi$-Rickart module is dual Rickart.

\begin{prop}\label{red}  Let $M$ be  a dual Rickart module. Then $M$ is dual
$\pi$-Rickart. The converse holds if $S$ is a reduced ring.
\end{prop}

\begin{proof}  The first statement is clear.  Suppose that $S$ is a reduced ring and
$M$ is a dual $\pi$-Rickart module. Let $f\in S$. There exist a
positive integer $n$ and  an idempotent $e\in S$ such that Im$f^n
= eM$. If $n=1$, there is nothing to do. Assume that $n>1$. Then
$(1-e)f^n M = 0$ and so  $(1-e)f^n = 0$. Since $S$ is a reduced
ring, $e$ is central and $((1-e)f)^n = 0$. Also it implies $(1-e)f
= 0$ or $f = ef$. Thus $Imf \leq eM$. The reverse inclusion
$eM\leq Imf$ follows from $eM = f^n M\leq f(f^{n-1})M \leq fM$.
Therefore $eM = Imf$ and $M$ is a dual Rickart module.
\end{proof}

By using a different condition on an endomorphism ring of a module
we show that a dual $\pi$-Rickart module is dual Rickart. To do
this we need the following lemma.

\begin{lem}\label{epim} Let $M$ be a module.
Then $M$ is  dual $\pi$-Rickart and $S$ is a domain if and only if
every nonzero element of $S$ is an epimorphism.
\end{lem}

\begin{proof} The sufficiency is clear. For the necessity,
let  $M$  be a dual $\pi$-Rickart module and $0\neq f\in S$. Then
there exist a positive integer $n$ and  an idempotent $e\in S$
such that Im$f^n = eM$. Hence $f^n= ef^n$. Since $S$ is a domain
and $f^n$ is nonzero, we have $e = 1$ and so Im$f^n = M$. This
implies that Im$f=M$. Thus $f$ is an epimorphism. \end{proof}

Recall that  a module $M$ has {\it $C_2$ condition} if any
submodule $N$ of $M$ which is isomorphic to a direct summand of
$M$ is a direct summand, while a module $M$ is said to have {\it
$D_2$ condition} if any submodule $N$ of $M$ with $M/N$ isomorphic
to a direct summand of $M$, then $N$ is a direct summand of $M$.
 In the next result we obtain relations between $\pi$-Rickart
and dual $\pi$-Rickart modules by using $C_2$ and $D_2$
conditions. An endomorphism $f$ of a module $M$ is called {\it
morphic} \cite{NC} if $M/fM\cong$ Ker$f$. The module $M$ is called
 {\it morphic} if every endomorphism of $M$ is morphic.

\begin{thm}\label{gizem} Let
$M$ be a module. Then we have the following.
\begin{enumerate}\item[{\rm (1)}]  If $M$ is a  dual
$\pi$-Rickart module with D$_2$ condition, then  it is $\pi$-Rickart.
\item[{\rm (2)}] If $M$ is a  $\pi$-Rickart module with
C$_2$ condition, then  it is dual $\pi$-Rickart.
\item[{\rm (3)}] If $M$ is projective morphic, then it is
$\pi$-Rickart if and only if it is dual
$\pi$-Rickart.
\end{enumerate}
\end{thm}

\begin{proof} Since $M/$Ker$f^n \cong$ Im$f^n$,  D$_2$ and C$_2$ conditions complete the proof of (1) and
(2). The proof of (3) is clear.
\end{proof}

The next result is an immediate consequence of Theorem
\ref{gizem}.

\begin{cor} Let $M$ be a module with $C_2$ and $D_2$
conditions. Then  $M$ is a dual $\pi$-Rickart module if and only
if it is $\pi$-Rickart. \end{cor}

 In  \cite[Proposition 2.6]{LRR1}, it is shown that  $M$ is a dual
Rickart module if and only if the short exact sequence
$0\rightarrow Imf \rightarrow M \rightarrow M/${\rm Im}$f
\rightarrow 0$  splits for any $f\in S$.  In this direction we can
give a similar characterization for dual $\pi$-Rickart modules.

 \begin{lem}\label{split} The following are equivalent for a module $M$.
\begin{enumerate}
    \item[{\rm (1)}] $M$ is a dual $\pi$-Rickart module.
   \item[{\rm (2)}] For every $f\in S$ there exists a positive integer $n$ such that the short exact sequence
    $0\rightarrow {\rm  Im}f^{n}\rightarrow M \rightarrow M/${\rm Im}$f^{n}\rightarrow 0$
   splits.
\end{enumerate}
\end{lem}

\begin{proof} For any $f\in S$ and any positive integer $n$
consider the short exact sequence $0\rightarrow {\rm
Im}f^{n}\rightarrow M \rightarrow M/${\rm Im}$f^{n}\rightarrow 0$.
    The short exact sequence  splits in $M$
if and only if Im$f^n$ is a direct summand of $M$ if and only if
$M$ is  a dual $\pi$-Rickart module.
\end{proof}

One may suspect that every submodule of a dual $\pi$-Rickart
module is dual $\pi$-Rickart. The following example shows that
this is not the case.

\begin{ex}\label{alt}{\rm Consider $\Bbb Q$ as a $\Bbb Z$-module.
Then $S =$ End$_{\Bbb Z}(\Bbb Q)$ is isomorphic to $\Bbb Q$. Since
every element of $S$ is an isomorphism or zero, $\Bbb Q$ is dual
$\pi$-Rickart. Now consider the submodule $\Bbb Z$ and $f\in$
End$_{\Bbb Z}(\Bbb Z)$ defined by $f(x) = 2x$,  where $x\in \Bbb
Z$. Since the image of any power of $f$ can not be a direct
summand of $\Bbb Z$, the submodule $\Bbb Z$ is not dual
$\pi$-Rickart. }
\end{ex}

Although  every submodule of a dual $\pi$-Rickart module need not
be  dual $\pi$-Rickart by Example \ref{alt},  we now prove that
every direct summand of  dual $\pi$-Rickart modules is also dual
$\pi$-Rickart.

\begin{prop}\label{diksub} Let $M$ be a  dual $\pi$-Rickart module. Then every direct summand
of $M$ is also dual $\pi$-Rickart.
\end{prop}
\begin{proof} Let $M = N\oplus P$ with  $S_{N}=$ End$_{R}(N)$. Define $g =
f\oplus 0_{|_{P}}$, for any $f\in S_{N}$ and so $g\in S$. By
hypothesis, there exist a positive integer $n$ and $e^2 = e\in S$
such that Im$g^n = eM$ and $g^n = f^n \oplus 0_{|_{P}}$. Hence
\linebreak $eM = $ Im$g^n = f^nN\leq N$. Let $M = eM\oplus Q$ for
some submodule $Q$. Thus $N = eM \oplus (N\cap Q) = f^nN\oplus
(N\cap Q)$. Therefore $N$ is dual $\pi$-Rickart.
\end{proof}

\begin{cor}\label{smndcor} Let $R$ be a
$\pi$-regular ring  with  $e=e^2\in R$. Then $M=eR$ is a dual
$\pi$-Rickart $R$-module.
 \end{cor}

 Here we give the following result for $\pi$-regular rings.

\begin{cor} Let $R = R_1\oplus R_2$ be a $\pi$-regular ring with direct sum of the rings $R_1$ and $R_2$. Then
the rings $R_1$ and $R_2$ are also $\pi$-regular.
\end{cor}

We now characterize $\pi$-regular rings in terms of dual
$\pi$-Rickart modules.
\begin{thm} Let $R$ be a ring. Then $R$ is $\pi$-regular if and
only if every cyclic projective $R$-module is dual $\pi$-Rickart.
\end{thm}
\begin{proof} The sufficiency is clear. For the necessity, let
$M=mR$ be a projective module. Then $R=r_R(m)\oplus I$ for some
right ideal $I$ of $R$. Let $I \stackrel{\varphi}\rightarrow M$
denote the isomorphism and $f\in S$. By Proposition \ref{pi} and
Proposition \ref{diksub},
$(\varphi^{-1}f\varphi)^nI=(\varphi^{-1}f^n\varphi)I$ is a direct
summand of $I$ for some positive integer $n$. Hence
$I=(\varphi^{-1}f^n\varphi)I\oplus K$ for some  right ideal $K$ of
$I$. Thus $\varphi I=(f^n \varphi)I\oplus \varphi K$,  and so
$M=f^nM\oplus \varphi K$. Therefore $M$ is dual $\pi$-Rickart.
\end{proof}

\begin{thm} Let $R$ be a ring and consider the following conditions.
\begin{enumerate}
    \item [{\rm (1)}] Every free $R$-module is
   dual  $\pi$-Rickart.
    \item [{\rm (2)}] Every projective $R$-module is
    dual $\pi$-Rickart.
     \item [{\rm (3)}] Every flat $R$-module is
    dual $\pi$-Rickart.
\end{enumerate}
Then {\rm (3) $\Rightarrow$ (2) $\Leftrightarrow$ (1)}. Moreover
{\rm (2) $\Rightarrow$ (3)} holds for finitely presented modules.
\end{thm}
\begin{proof} (3) $\Rightarrow$  (2) $\Rightarrow $ (1) Clear. (1) $\Rightarrow $ (2)
Let $M$ be a projective $R$-module. Then $M$ is a direct summand
of a free $R$-module $F$. By (1), $F$ is dual  $\pi$-Rickart, and
so is $M$ due to Proposition \ref{diksub}.

(2) $\Rightarrow$ (3) is clear from the fact that finitely
presented flat modules are projective.
\end{proof}

The next example reveals that a direct sum of dual $\pi$-Rickart
modules need not be dual $\pi$-Rickart.

\begin{ex}\label{orn3}{\rm Let $R$ denote the  ring $ \left(\begin{array}{ll}\Bbb
R&\Bbb R\\ 0&\Bbb R
\end{array}\right)$ and $M$ the $R$-module
$ \left (\begin{array}{cc}\Bbb R & \Bbb R\\\Bbb R & \Bbb R
\end{array}\right)$.
Let $f\in S$. Then there exist $a, c, u, t\in R$ such that
$f\left(\begin{array}{cc}x & y\\r & s\end{array}\right)=
\left(\begin{array}{cc}ax+ur & ay+us\\cx+tr &
cy+ts\end{array}\right)$ where $\left(\begin{array}{cc}x & y\\r &
s\end{array}\right)\in M$. Consider $f\in S$ defined by
 $a = c = 0$, $u = 3$ and $t = 2$. This implies  that
 $f\left(\begin{array}{cc}x & y\\r &
s\end{array}\right) = \left(\begin{array}{cc}3r & 3s\\2r &
2s\end{array}\right)$ and for any positive integer
 $n$ we obtain \begin{center} $f^n\left(\begin{array}{cc}x & y\\r &
s\end{array}\right) = \left(\begin{array}{cc}3(2^{n-1})r &
3(2^{n-1})s\\2^nr & 2^ns\end{array}\right)$\end{center} It follows
that $f^n M$ can not be a direct summand. On the other hand,
consider the submodules $N = \left(\begin{array}{ll}\Bbb R&\Bbb R\\
0&0\end{array}\right)$ and  $K = \left(\begin{array}{ll}0 & 0\\
\Bbb R & \Bbb R\end{array}\right)$ of $M$. Then End$_R(N)$ and
End$_R(K)$ are isomorphic to  $\Bbb R$. Hence $N$ and $K$ are dual
$\pi$-Rickart modules but $M$ is not dual $\pi$-Rickart.}
\end{ex}

The following lemma is useful to show that a direct sum of some
dual $\pi$-Rickart modules is dual $\pi$-Rickart.

\begin{lem}\label{abel} Let $M$ be a module  and $f\in S$. If Im$f^n = eM$ for
 some central idempotent
$e \in S$ and  a positive integer  $n$, then Im$f^{n+1}=eM$.
\end{lem}
\begin{proof} Let $f\in S$ and  Im$f^n = eM$ for some central idempotent
$e \in S$ and  a positive integer  $n$. It is clear that
Im$f^{n+1}\subseteq$ Im$f^{n}$. Let  $f^n(x) \in$ Im$f^{n}$,
 then $f^{n}(x) = e f^{n}(x) = f^{n}e(x)$. Since $e(x)\in$ Im$f^{n}$,
 $e(x) = f^{n}(y)$ for some $y\in M$. So $f^{n}(x) = f^{n}(f^{n}(y))=f^{n+1}(f^{n-1}(y))\in$
  Im$f^{n+1}$. This
completes the proof.
\end{proof}

A ring $R$ is called {\it abelian} if every idempotent is central,
that is, $ae=ea$ for any $a, e^2=e  \in R$.  A module $M$ is
called {\it abelian} \cite{Ro} if $fem = efm$ for any $f\in S$,
$e^2 = e\in S$, $m\in M$. Note that $M$ is an abelian module if
and only if $S$ is an abelian ring. We now prove that a direct sum
of dual $\pi$-Rickart modules is dual $\pi$-Rickart for some
abelian modules.

\begin{prop} Let $M_1$ and $M_2$ be abelian $R$-modules.  If
$M_1$ and $M_2$ are dual $\pi$-Rickart with
 Hom$_R(M_i, M_j)=0$ for $i\neq j$, then $M_1\oplus M_2$ is a
dual $\pi$-Rickart module.
\end{prop}
\begin{proof} Let  $S_1=$ End$_R(M_1)$, $S_2=$ End$_R(M_2)$ and $M=M_1\oplus M_2$.
 We may describe $S$ as $\left(\begin{array}{cc}S_1 &
0\\0 & S_2\end{array}\right)$. Let $\left(\begin{array}{cc}f_1 &
0\\0 & f_2\end{array}\right)\in S$, where $f_1\in S_1$ and $f_2\in
S_2$. Then there exist positive integers $n, m$ and $e_1^2=e_1\in
S_1$ and $e_2^2=e_2\in S_2$ such that Im$f_1^n=e_1M_1$ and
Im$f_2^m=e_2M_2$. Consider the following cases:

 ({\it i}) Let $n=m$. Obviously, Im$\left (\begin{array}{cc}f_1 &
0\\0 & f_2\end{array}\right)^n=\left(\begin{array}{cc}e_1 & 0\\0 &
e_2\end{array}\right)M$.

({\it ii}) Let $n<m$. By Lemma \ref{abel}, we have
Im$f_1^n=$Im$f_1^m=e_1M_1$. Clearly,
 $\left(\begin{array}{cc}e_1 & 0\\0 &
e_2\end{array}\right)M\leq$  Im$\left(\begin{array}{cc}f_1 & 0\\0
&
f_2\end{array}\right)^m$. Now let $\left(\begin{array}{c} m_1\\
m_2\end{array}\right)\in$ Im$\left( \begin{array}{cc}f_1 & 0\\0 &
f_2\end{array}\right)^m$. Then  $m_1\in$ Im$f_1^m=e_1M_1$ and
$m_2\in$ Im$f_2^m=e_2M_2$.  Hence  \linebreak $\left(\begin{array}{c} m_1\\
m_2\end{array}\right)= \left(\begin{array}{cc}e_1 & 0\\0 &
e_2\end{array}\right)\left(\begin{array}{c} m_1\\
m_2\end{array}\right)$. Thus $\left(\begin{array}{c} m_1\\
m_2\end{array}\right)\in \left(\begin{array}{cc}e_1 & 0\\0 &
e_2\end{array}\right)M$. Therefore Im$\left(
\begin{array}{cc}f_1 & 0\\0 & f_2\end{array}\right)^m \leq
\left(\begin{array}{cc}e_1 & 0\\0 & e_2\end{array}\right)M$.

({\it iii})  Let $m<n$. Since $M_2$ is abelian, the proof is
similar to case (ii).
 \end{proof}

We close this section with the relations among strongly co-Hopfian
modules, Fitting modules and dual $\pi$-Rickart modules.

Recall that a module $M$ is called {\it co-Hopfian}  if every
injective endomorphism of M is an automorphism, while $M$ is
called {\it strongly co-Hopfian} \cite{HKS}, if for any
endomorphism $f$ of $M$ the descending chain
\begin{center}
Im$f \supseteq $ Im$f^2 \supseteq \cdots \supseteq $  Im$f^n
\supseteq \cdots $ \end{center} stabilizes.

We now give a relation between abelian and strongly co-Hopfian
modules by using dual $\pi$-Rickart modules.

\begin{cor}\label{stcohop}  Let $M$ be a  dual $\pi$-Rickart module and $S$  an
abelian ring. Then $M$ is strongly co-Hopfian.
\end{cor}

\begin{proof} It follows from Lemma \ref{abel} and
\cite[Proposition 2.6]{HKS}.
\end{proof}

A module $M$ is said to be {\it a Fitting module} \cite{HKS} if
for any $f \in S$, there exists an integer $n \geq 1$ such that
$M=$ Ker$f^n\oplus $ Im$f^n$. Due to Armendariz, Fisher and Snider
\cite{AFS} or \cite[Proposition 5.7]{Tu}, the module $M$ is
Fitting if and only if $S$ is  strongly $\pi$-regular.

We now give the following relation between Fitting modules and
dual $\pi$-Rickart modules.
\begin{cor}\label{fgric} Every Fitting module is a dual
$\pi$-Rickart module.
\end{cor}

Then we have the following result.

\begin{prop}\label{artin} Let $R$ be an Artinian ring.
Then every finitely generated $R$-module is dual $\pi$-Rickart.
\end{prop}
\begin{proof} Let  $M$ be  a finitely
generated $R$-module. Then $M$ is an Artinian and Noetherian
module. Hence $M$ is a Fitting module and so it is dual
$\pi$-Rickart.
\end{proof}

\begin{prop} Let $R$ be a ring and $n$ a positive integer. If the
matrix ring $M_n(R)$ is strongly $\pi$-regular, then $R^n$ is a
dual $\pi$-Rickart $R$-module.
\end{prop}
\begin{proof} Let $M_n(R)$ be a strongly $\pi$-regular ring. Then
by \cite[Corollay 3.6]{HKS}, $R^n$ is a Fitting $R$-module and so
it is dual $\pi$-Rickart.
\end{proof}

\section{The Endomorphism Ring of a dual $\pi$-Rickart Module}

In this section we study relations between a dual $\pi$-Rickart
module and its endomorphism ring. We prove that the endomorphism
ring of a dual $\pi$-Rickart module is always a generalized left
principally projective ring, the converse holds if the module is
self-cogenerator. The modules whose endomorphism rings are
$\pi$-regular are characterized. It is shown that if the module
satisfies $D_2$ condition, then it is dual $\pi$-Rickart if and
only if the endomorphism ring of the module is a $\pi$-regular
ring.

\begin{lem}\label{grpp} If $M$ is a dual $\pi$-Rickart module, then $S$ is a generalized
left principally projective ring.
\end{lem}

\begin{proof} Let $f\in S$. By assumption, there exist $e^{2}=e\in S$ and
 a positive integer $n$  such that Im$f^{n}=eM$. Hence $l_{S}(f^{n}M)=S(1-e)=l_{S}(f^{n})$.
Thus $S$ is a  generalized left principally projective ring.
\end{proof}

The next result is a consequence of Theorem \ref{diksub} and Lemma
\ref{grpp}.

\begin{cor} If $R$ is a $\pi$-regular ring,
then $eRe$ is a generalized left principally projective ring for
any $e^2=e\in R$.
\end{cor}

\begin{cor} Let $M$ be a dual $\pi$-Rickart module and $f\in S$. Then $Sf^n$ is a projective left
$S$-module for some positive integer $n$.
\end{cor}

\begin{proof} Clear from Lemma \ref{grpp}, since $Sf^n \cong
S/l_S(f^n)$. \end{proof}

Recall that a module is called {\it self-cogenerator} if it
cogenerates all its factor modules. The following result shows
that the converse of Lemma \ref{grpp} is true for self-cogenerator
modules. On the other hand, Theorem \ref{wis} generalizes the
result \cite[39.11]{Wi1}.

\begin{thm}\label{wis} Let $M$ be a module and $f\in S$.
\begin{enumerate}
\item[(1)] If $Sf^n$ is a projective left $S$-module for some positive integer $n$,
then the submodule
$N = \bigcap~ \{$Ker$g\mid g\in S,$ Im$f^n\leq$  Ker$g\}$ is a
direct summand of $M$.
\item[(2)] If $M$ is self-cogenerator and $S$ is a  generalized left
principally projective ring,
then $M$ is a dual $\pi$-Rickart module.
\end{enumerate}
\end{thm}
\begin{proof} (1) Let $Sf^n$ be a projective left $S$-module for some positive integer $n$. Since $Sf^n\cong S/l_S(f^n)$,
$l_S(f^n) = Se$ for some $e^2 = e\in S$. We prove $(1 - e )M = N$.
Due to $ef^nM = 0$, we have $f^n M\leq (1 - e)M$. By definition of
$N$ we have $N\leq (1 - e)M$. Let $g\in S$ with Im$f^n \leq$
Ker$g$. Then $gf^nM = 0$ or $gf^n = 0$. Hence $g\in l_S(f^n) = Se$
and $ge = g$. So $g(1 - e)M = 0$ from which we have $(1 - e)M
\leq$ Ker$g$ for all $g$
with Im$f^n  \leq $ Ker$g$. Thus  $(1 - e)M\leq N$. Therefore $(1 - e)M = N$. \\
(2) Assume that $M$ is self-cogenerator and $S$ is generalized
left principally projective. There  exist $e^{2}=e\in S$,   a
positive integer $n$  such that $l_S(f^n) = Se$ and
 $M/$Im$f^n$ is cogenerated by $M$. By \cite[14.5]{Wi1},
\begin{center} $\bigcap~\{$ Ker$g\mid g\in$ Hom$(M/$Im$f^n, M)\} = 0$.
\end{center} Hence \begin{center} Im$f^n =  \bigcap~\{$ Ker$g\mid g\in
S,$ Im$f^n\leq$  Ker$g \}$.\end{center} Thus conditions of (1) are
satisfied and so Im$f^n$ is a direct summand.
\end{proof}

For an $R$-module $M$, it is shown that, if $S$ is a von Neumann
regular ring, then $M$ is a dual Rickart module (see
\cite[Proposition 3.8]{LRR1}). We obtain a similar result for dual
$\pi$-Rickart modules.

\begin{lem}\label{pireg}  Let $M$ be a  module. If $S$ is a $\pi$-regular ring, then $M$ is  dual
$\pi$-Rickart.
\end{lem}

\begin{proof}  Let $f\in S$. Since $S$ is $\pi$-regular, there exist a positive integer $n$ and $g\in S$
such that $f^n=f^ngf^n$. Then  $e=f^ng$ is an idempotent of $S$.
Now we show that Im$f^n=f^ngM$. It is clear that $f^{n}M=ef^{n}M\leq eM$. For the other inclusion, let $m\in M$. Hence $em=f^{n}gm\in f^{n}M$. Thus Im$f^{n}=eM$.
\end{proof}

Since every strongly $\pi$-regular ring is $\pi$-regular, we have
the next result.
\begin{cor}\label{prg}  Let $M$ be a  module. If $S$ is a strongly $\pi$-regular ring,
then $M$ is  dual
$\pi$-Rickart.
\end{cor}

The converse statement of Corollary \ref{prg} does not hold in
general, that is,  there exists a dual $\pi$-Rickart module having
not a strongly $\pi$-regular endomorphism ring.

\begin{ex} {\rm Let $D$ be a division ring, $M$ a vector space over $D$ with an
infinite basis $\{e_i\in M\mid i = 1,2,...\}$ and  $S =$
End$_D(M)$. As a semisimple right $D$-module, $M$ is dual
$\pi$-Rickart, and by \cite[3.9]{Wi1} $S$ is a regular and so
$\pi$-regular ring. Assume that $S$ is a strongly $\pi$-regular
ring and we reach a contradiction. Let $f\in S$ defined by $f(e_i)
= e_{i+1}$ for all $i = 1,2,3,....$ By assumption, there is a
positive integer $n$ such that $f^n = f^{n+1}g$ for some $g\in S$.
Then $f^n = f^{n+1}g$ implies $f^nS = f^{n+1}S$ and so $f^n M =
f^{n+1}M$. Since $f^n(e_i) = e_{i+n}$ for all $i$, we have $f^n M
= \sum\limits_{i>n}e_iD \neq f^{n+1} M$. This is a contradiction.
Hence $S$ is not a strongly $\pi$-regular ring (see also
\cite[5.5]{Tu}).}
\end{ex}

The proof of Lemma \ref{maybe} may be in the context.
\begin{lem}\label{maybe} Let $M$ be a module.
Then $S$ is a $\pi$-regular ring if and only if there exists a
positive integer $n$ such that Ker$f^n$ and Im$f^n$ are direct
summands of $M$ for any $f\in S$.
\end{lem}

Now we  recall some known facts that will be needed about
$\pi$-regular rings.
\begin{lem}\label{dorsey} Let $R$ be a ring. Then
\begin{enumerate}
    \item[{\rm (1)}] If $R$ is $\pi$-regular, then $eRe$ is also
    $\pi$-regular  for any $e^2=e\in R$.
    \item[{\rm (2)}] If $M_n(R)$ is $\pi$-regular for any positive integer
    $n$, then so is $R$.
    \item[{\rm (3)}] If $R$ is a commutative ring, then $R$ is $\pi$-regular
    if and only if $M_n(R)$ is $\pi$-regular for any positive integer
    $n$.
\end{enumerate}
\end{lem}
\begin{prop}\label{fgproj}
Let $R$ be a commutative $\pi$-regular ring. Then every finitely
generated projective $R$-module is dual $\pi$-Rickart.
\end{prop}
\begin{proof} Let $M$ be a finitely generated  projective $R$-module.
So the endomorphism ring of
$M$ is $eM_n(R)e$ with some positive integer $n$ and an idempotent
$e$ in $M_n(R)$. Since $R$ is commutative  $\pi$-regular, $M_n(R)$
is also  $\pi$-regular, and so is $eM_n(R)e$ by Lemma
\ref{dorsey}. Hence $M$ is dual $\pi$-Rickart  by Lemma
\ref{pireg}.
\end{proof}

\begin{thm}\label{D2}  Let $M$ be a  module with $D_2$ condition.
Then $M$ is dual $\pi$-Rickart if
and only if $S$ is  $\pi$-regular.
\end{thm}
\begin{proof}  The necessity holds by Lemma \ref{pireg}. For the
sufficiency, let  $0\neq f\in S$. Since $M$ is dual $\pi$-Rickart,
Im$f^n$ is a direct summand of $M$ for some positive integer $n$.
Because of $M/$Ker$f^n \cong $ Im$f^n$, D$_2$ condition implies
that Ker$f^n$ is a direct summand of $M$. The rest is obvious from
Lemma \ref{maybe}.
\end{proof}

The following is a  consequence of Proposition \ref{fgproj} and
Theorem \ref{D2}.
\begin{cor} Let $R$ be a commutative ring and satisfy $D_2$
condition. Then the following are equivalent.\begin{enumerate}
    \item [{\rm (1)}] $R$ is a $\pi$-regular ring.
    \item [{\rm (2)}] Every finitely generated projective
    $R$-module is dual $\pi$-Rickart.
\end{enumerate}
\end{cor}

Recall that a module $M$ is called {\it quasi-projective} if  it
is $M$-projective. Since every quasi-projective module has $D_2$
condition, we have the following.

\begin{cor}\label{prgd} If $M$ is a quasi-projective dual $\pi$-Rickart module,
then the endomorphism ring  of $M$ is a $\pi$-regular ring.
\end{cor}

\begin{thm}  The following are equivalent for a ring $R$.
\begin{enumerate}
    \item[{\rm (1)}] $M_n(R)$ is $\pi$-regular for every positive
integer $n$.
   \item[{\rm (2)}] Every finitely generated projective $R$-module is
   dual $\pi$-Rickart.
\end{enumerate}
\end{thm}
\begin{proof} (1) $\Rightarrow$ (2) Let $M$ be a finitely generated
projective $R$-module. Then $M\cong eR^n$ for some positive
integer $n$ and $e^2=e\in M_n(R)$. Hence $S$ is isomorphic to
$eM_n(R)e$. By (1), $S$ is $\pi$-regular. Thus $M$ is
$\pi$-Rickart due to Lemma \ref{pireg}.

(2) $\Rightarrow$ (1) $M_n(R)$ can be viewed as the endomorphism
ring of a projective $R$-module $R^n$ for any positive integer
$n$. By (2), $R^n$ is dual $\pi$-Rickart. Then $M_n(R)$ is
$\pi$-regular by Corollary \ref{prgd}.
\end{proof}

Recall that an $R$-module $M$ is called {\it duo} if every
submodule of $M$ is fully invariant, i.e., for any submodule $N$
of $M$, $f N\leq N$ for each  $f\in S$, equivalently, every right
$R$-submodule of $M$ is also left $S$-submodule. Our next aim is
to determine to find conditions under which any factor module of a
dual $\pi$-Rickart module is also dual $\pi$-Rickart.

\begin{cor} Let $M$ be a quasi-projective module and $N$ a fully invariant submodule of $M$. If $M$ is
dual $\pi$-Rickart, then so is $M/N$.
\end{cor}
\begin{proof} Let $f\in S$ and $\pi$ denote the natural
epimorphism from $M$ to $M/N$. Consider the following diagram.
$$\xymatrix{ M \ar[d]_f \ar[r]^\pi & M/N  \ar@{.>}[d]^{f^*} \\
M \ar[r]_\pi & M/N\\
 }$$
 Since $N$ is fully invariant, we have Ker$\pi\subseteq$ Ker$\pi
 f$.  By the Factor Theorem, there exists a unique homomorphism
 $f^*$ such that $f^* \pi=\pi f$. Hence we define a homomorphism
 $\varphi : S\rightarrow$ End$_R(M/N)$ with $\varphi(f)=f^*$ for
 any $f\in S$. As $M$ is quasi-projective, $\varphi$ is an
 epimorphism. Thus End$_R(M/N)\cong S/$Ker$\varphi$. By Corollary
 \ref{prgd},  $S$ is  $\pi$-regular, and so is
 $S/$Ker$\varphi$. Therefore $M/N$ is dual $\pi$-Rickart due to Lemma
 \ref{pireg}.
\end{proof}

\begin{cor} Let $M$ be a quasi-projective duo module. If $M$ is  dual $\pi$-Rickart,
then $M/N$ is also dual $\pi$-Rickart for every submodule $N$ of
$M$.
\end{cor}

\begin{cor} If $M$ be a quasi-projective dual $\pi$-Rickart
module, then so is $M/$Rad$(M)$ and $M/$Soc$(M)$.
\end{cor}

\begin{prop}\label{image} Let $M$ be a dual $\pi$-Rickart module. Then every
endomorphism of $M$ with a small image in $M$ is nilpotent.
\end{prop}
\begin{proof} Let $f\in S$ with Im$f$ small in $M$. Then Im$f^n$
is a direct summand of $M$ for some positive integer $n$. Also
Im$f^n$ is small in $M$. Hence $f^n=0$.
\end{proof}

\begin{cor} Let $M$ be a dual $\pi$-Rickart discrete module. Then
$J(S)$ is nil and $S/J(S)$ is von Neumann regular.
\end{cor}
\begin{proof} Since $M$ is discrete, by \cite[Theorem 5.4]{MM}, $J(S)$ consists of
endomorphisms with small image. By Proposition \ref{image}, $J(S)$
is nil and again by \cite[Theorem 5.4]{MM}, $S/J(S)$ is von
Neumann regular.
\end{proof}

\begin{thm}\label{gds}  The following are equivalent for a module $M$.
\begin{enumerate}
    \item [{\rm (1)}] $M$ is a dual $\pi$-Rickart module.
    \item [{\rm (2)}] $S$ is a  generalized left principally projective ring  and
$f^{n}M=r_{M}(l_{S}(f^{n}M))$ for all $f \in S$ and a positive
integer $n$.
\end{enumerate}
\end{thm}

\begin{proof} (1) $\Rightarrow$ (2)  By Lemma \ref{grpp},  we only need to show
that   $f^{n}M=r_{M}(l_{S}(f^{n}M))$ for all $f \in S$. Since $M$
is dual $\pi$-Rickart, for any $f \in S$, $f^{n}M=eM$ for some
$e^{2}=e\in S$ and a positive integer $n$. Thus
$r_{M}(l_{S}(f^{n}M))=r_{M}(l_{S}(eM))=eM=f^{n}M$.\\
(2) $\Rightarrow$ (1) Let $f\in S$. Since $S$ is a generalized
left principally projective ring, $l_{S}(f^{n}M)=Se$ for some
$e^{2}=e\in S$ and a positive integer $n$. By hypothesis,
$f^{n}M=r_{M}(l_{S}(f^{n}M))=r_{M}(Se)=(1-e)M$. Thus $M$ is dual
$\pi$-Rickart.
\end{proof}

\begin{cor} Let $M$ be a module. Then
 $M$ is  dual $\pi$-Rickart if and only if
$f^{n}M=r_{M}(l_{S}(f^{n}M))$ and $r_{M}(l_{S}(f^{n}M))$ is a
direct summand of $M$.
\end{cor}

\begin{thm} Let $M$ be a dual $\pi$-Rickart module.
Then the left singular ideal $Z_l(S)$ of $S$ is nil
and $Z_l(S)\subseteq J(S)$.
\end{thm}
\begin{proof} Let $f\in Z_l(S)$. Since $M$ is dual $\pi$-Rickart,
Im$(f^n)=eM$ for some positive integer $n$ and $e=e^2\in S$. Then,
by Lemma \ref{grpp}, $l_S(f^n)=S(1-e)$. Since $l_S(f^n)$ is
essential in $S$ as a left ideal, we have $l_S(f^n)=S$. This
implies that $f^n=0$ and so $Z_l(S)$ is nil. On the other hand,
for any $g\in S$ and $f\in Z_l(S)$,  according to previous
discussion, $(gf)^n=0$ for some positive integer $n$. Hence $1-gf$
is invertible. Thus $f\in J(S)$. Therefore $Z_l(S)\subseteq J(S)$.
\end{proof}

\begin{prop}\label{local} The following are equivalent for a module $M$.
\begin{enumerate}
\item[{\rm (1)}] $M$ is an indecomposable  dual $\pi$-Rickart
module.
\item[{\rm (2)}] Each element of $S$ is either an epimorphism or nilpotent.
\end{enumerate}
\end{prop}
\begin{proof} (1) $\Rightarrow$ (2)  Let $f\in S$. Then $f^nM$ is a direct
summand of $M$ for some positive integer $n$. As $M$ is
indecomposable, we see that $f^nM=0$ or $f^nM=M$. This implies
that $f$ is an epimorphism or nilpotent.

(2) $\Rightarrow$ (1) Let $e=e^2\in S$. If $e$ is nilpotent, then
$e=0$. If $e$ is an epimorphism, then $e=1$. Hence $M$ is
indecomposable. Also for any $f\in S$, $fM=M$ or $f^nM=0$ for some
positive integer $n$. Therefore $M$ is dual $\pi$-Rickart.
\end{proof}

\begin{thm}\label{pimorphic} Consider the following conditions for a module $M$.
\begin{enumerate}
    \item[{\rm (1)}] $S$ is a local ring with nil Jacobson radical.
   \item[{\rm (2)}] $M$ is an indecomposable dual $\pi$-Rickart
   module.
\end{enumerate}
Then {\rm (1)} $\Rightarrow$ {\rm (2)}. If $M$ is a morphic
module, then {\rm (2)} $\Rightarrow$ {\rm (1)}.
\end{thm}
\begin{proof} (1) $\Rightarrow$ (2) Clearly, each element of $S$ is either an epimorphism
or nilpotent. Then $M$ is  indecomposable dual $\pi$-Rickart due
to Proposition \ref{local}.

(2) $\Rightarrow$ (1)  Let $f\in S$. Then $f^nM=eM$ for some
positive integer $n$ and an idempotent $e$ in $S$. If $e=1$, then
$f$ is an epimorphism. Since $M$ is morphic, $f$ is  invertible by
\cite[Corollary 2]{NC}. If $e=0$, then $f^n=0$. Hence $1-f$ is
invertible. This implies that $S$ is a local ring. Now let $0\neq
f\in J(S)$. Since $f$ is not invertible and $M$ is morphic, $f$ is
nilpotent by Proposition \ref{local}. Therefore $J(S)$ is nil.
\end{proof}

 The next result can be obtained from Theorem \ref{pimorphic} and
\cite[Lemma 2.11]{HC}.
\begin{cor} Let $M$ be an indecomposable dual $\pi$-Rickart module. If
$M$ is morphic, then $S$ is a left and right $\pi$-morphic ring.
\end{cor}

\noindent{\bf Acknowledgements.} The authors would like to thank
the referee(s) for valuable suggestions. The first author thanks
the Scientific and Technological Research Council of Turkey
(TUBITAK) for the financial support.

\end{document}